\documentclass{amsart}
\usepackage{amsmath,amscd,amsfonts,amssymb}
\usepackage{xcolor}
\usepackage{bm}
\usepackage{enumerate}

 \newtheorem{theorem}{Theorem}[section]
 \newtheorem{corollary}[theorem]{Corollary}
 \newtheorem{lemma}[theorem]{Lemma}
 \newtheorem{proposition}[theorem]{Proposition}
 \newtheorem{definition}[theorem]{Definition}

\theoremstyle{remark}

 \newtheorem{remark}[theorem]{Remark}
 \newtheorem{notation}[theorem]{Notation}

 \newtheorem{example}[theorem]{Example}
 \numberwithin{equation}{section}

\newcommand{\lie}[1]{\mathfrak L (#1)}

\newcommand{\alg}[1]{k\langle#1\rangle}
\newcommand{\inspitz}[1]{\langle\!\langle#1\rangle\!\rangle}
\newcommand{\pow}[1]{k\inspitz{#1}}
\DeclareMathOperator{\lk}{span}
\newcommand{\linkomb}[1]{\lk\{#1\}}
\newcommand{\co}[4]{[\uppercase{#1}_{#3},\uppercase{#2}_{#4}]}
\newcommand{\cco}[6]{[\uppercase{#1}_{#4},\co{#2}{#3}{#5}{#6}]}

\newcommand\NN{\mathbb N}
\newcommand{\dd}{{\rm d}}
\newcommand\ddt[1]{\frac{\dd^{#1}}{\dd t^{#1}}}

\newcommand{\wh}[1]{\textcolor{green!60!black}{#1}}

\newcommand{\ignore}[1]{}

\long\def\alert#1{\parindent2em\smallskip\hbox to\hsize%
{\hskip\parindent\vrule%
\vbox{\advance\hsize-2\parindent\hrule\smallskip\parindent.4\parindent%
\narrower\noindent\wh{#1}\smallskip\hrule}\vrule\hfill}\smallskip\parindent0pt}

\title[]{The BCH-Formula and Order Conditions for Splitting Methods}

\author[W. Auzinger]{Winfried Auzinger}
\address{\newline
Institute for Analysis and Scientific Computing  \newline
Technische Universit\"at Wien\newline
Wiedner Hauptstra\ss e 8-10/E101\newline
Vienna, Austria}
\email{Winfried.Auzinger@tuwien.ac.at}

\author[W. Herfort]{Wolfgang Herfort}\thanks{Wolfgang Herfort would like to thank the Department of Mathematics at the Brigham Young University for the great hospitality during the year 2015. Special thanks to Jim Logan for his excellent support with hard- and software.}
\address{\newline
Institute for Analysis and Scientific Computing  \newline
Technische Universit\"at Wien\newline
Wiedner Hauptstra\ss e 8-10/E101\newline
Vienna, Austria}
\email{Wolfgang.Herfort@tuwien.ac.at}

\author[O. Koch]{Othmar Koch}\thanks{Othmar Koch acknowledges the support by the Vienna Science and Technology Fund (WWTF) under the grant MA14-002.}
\address{\newline
Fakultat f\"ur Mathematik, Universit\"at Wien \newline
Oskar-Morgenstern-Platz 1 \newline
Vienna, Austria}
\email{othmar@othmar-koch.org}

\author[M. Thalhammer]{Mechthild Thalhammer}
\address{\newline
Institut f\"ur Mathematik \newline
Universit\"at Innsbruck \newline
Technikerstra\ss e 13 \newline
Innsbruck, Austria}
\email{Mechthild.Thalhammer@uibk.ac.at}

\begin{document}
\maketitle
\begin{abstract}
As an application of the BCH-formula, order conditions for splitting schemes are derived.
The same conditions can be obtained by using non-commutative power series techniques
and inspecting the coefficients of Lyndon-Shirshov words.
\end{abstract}
\section{Introduction}
The main purpose of this note is to present a not so well-known application of the
Baker-Campbell-Hausdorff formula (BCH-formula): Computing {\em order conditions} for
{\em exponential splitting schemes}.
There is vast literature, for an overview we particularly refer to~\cite{mclachlanetal92} and \cite[Chapter~III]{haireretal06}.

The topic of splitting is a comparatively young field and it is our intention to
present only facets  -- with Lie-theoretic background.
We shall first recall a few facts from Lie theory and power series as far as needed.
An abstract definition of splitting is given and the computation of order condition is
demonstrated with examples. The last section is devoted to an alternative approach for finding
the order conditions by inspecting the coefficients of leading Lyndon-Shirshov words
in an exponential function of a sum of Lie elements, as currently used by the authors for
computationally generating order conditions for exponential splitting schemes.

\section{Formal Power Series}
Let $k$ be a field of characteristic zero.
Then $\pow S$ denotes the algebra of formal power series with coefficients
in $k$ -- $S$ a set of non-commuting variables. The natural grading of $\pow S$ is given as follows:
Elements in $k$ have degree zero,
those in $S^n:=\{s_1\cdots s_n\mid s_j\in S\}$ have degree $n$, where $n\ge1$.
A {\em homogeneous}
element in $\pow S$ is a $k$-linear combination of elements of the same degree.
Every element $f$ in $\pow S$ allows a unique decomposition into  homogeneous components
\[f=\sum_{j=0}^{\infty} f_j \]
where for each $j$ the element $f_j$ is homogeneous of degree $j$.

In our context the power series ring $R\inspitz t$ in a single variable $t$ and with coefficients in
a (not necessarily commutative) ring  $R$ will turn out to be useful.

Whenever $f=\sum_j c_j t^j\in\pow t$ and $g\in \pow S$ and $g$ does {\em not}
contain constant terms then one can define the composition
\[f\circ g:=\sum_j c_j g^j \]
as, for given degree say $n$, for $j\ge n+1$, the power series $g^j$ does not contribute
homogeneous elements of degree $n$.

\begin{example}\label{ex:exponential}
The univariate formal power series $f:=\sum_{j=0}^{\infty} \frac{1}{j!}t^j$ will be denoted by
$e^t$. Hence the composition $f\circ g$ allows to consider $e^g=\sum_{j=0}^{\infty}\frac{1}{j!}g^j$.
\end{example}

The following simple fact will be helpful:

\begin{lemma}\label{l:homog}
Let $h=\sum_{j=1}^{\infty} h_j$ be an element in $\pow S$,
with each $h_j$ homogeneous of degree $j$.
Then $e^h=\sum_{j=0}^{\infty} e_j$ with homogeneous terms $e_j$,
and the following statements are equivalent:
\begin{enumerate}[{\rm (i)}]
\item $h_j=0$ for $j=1,\ldots,p$;
\item $e_j=0$ for $j=1,\ldots,p$.
\end{enumerate}
\end{lemma}

\begin{proof}
Suppose that (i) holds. Then
\[e^h=1 + \big(h_{p+1}+\cdots\big) + \frac{1}{2}\big(h_{p+1}+\cdots\big)^2 + \cdots \]
shows that there cannot exist homogeneous terms $e_j$ with $1\le j\le p$.

The converse is proved by induction. Suppose for $p$ that $e_1=\cdots=e_{p-1}=0$ implies
$h_1=\cdots=h_{p-1}=0$.  Suppose next  that also $e_{p}=0$. Then
\[e^h=1+\big(h_p+\cdots\big)+\frac{1}{2}\big(h_p+\cdots\big)^2+\cdots=e_0+e_{p+1}+\cdots \]
From this one concludes that $h_p=0$ must hold.
\end{proof}

\begin{corollary}\label{c:homog}
Let $h=\sum_{j=1}^{\infty} h_j$ and $k=\sum_{j=1}^{\infty} k_j$ be elements in $\pow S$.
Set $e^h=\sum_{j=0}^{\infty} e_j$ and $e^k=\sum_{j=0}^{\infty} f_j$. Then the following statements
are equivalent
\begin{enumerate}[{\rm (a)}]
\item $h_j=k_j$ for $j=1,\ldots,p$;
\item $e_j=f_j$ for $j=1,\ldots,p$.
\end{enumerate}
\end{corollary}

\begin{proof}
Certainly (a) implies (b), as for forming the homogeneous terms in $e^h$ only the terms up to order $p$
contribute.

For proving the converse one again uses induction. Having established that $h_j=k_j$ for $1\le j\le p-1$,
one observes that $e_p=h_p+\phi(h_1,\ldots,h_{p-1})=k_p+\phi(k_1,\ldots,k_{p-1})=f_p$.
Here $\phi$ is a certain multivariate polynomial whose form to know is not needed.
Then, as $e_p=f_p$ conclude that $h_p=k_p$.
\end{proof}

\section{Reformulation Using Formal Differentiation}

Given the algebra $R:=\pow S$, one may use it as the set of coefficients and form the new algebra
$\pow S\inspitz t$.
There is a canonical function $\phi:\pow S\to \pow S\inspitz t$ that sends $f:=\sum_{j=0}^{\infty} f_j$
to the element $\phi(f):=\sum_{j=0}^{\infty} f_j t^j$.
In $\pow S\inspitz t$ we define {\em formal differentiation} by means of
\[ \Big(\sum_{j=0}^{\infty} r_j t^j\Big)^\cdot := \sum_{j=1}^{\infty} j r_{j-1} t^{j-1}. \]
Formal derivatives $f^{(k)}$ of higher order, for elements $f\in \pow S\inspitz t$,
are defined inductively.
\begin{notation}\label{n:order}\rm
If an element $f=\sum_{j=0}^{\infty} r_j t^j \in \pow S \inspitz t$
has $r_0=\cdots r_p=0$ we shall denote this by $f=O(t^{p+1})$ or even by $f=O(p+1)$.
\end{notation}

One proves without difficulty:

\begin{lemma}\label{l:order} For $f=\sum_{j=0}^{\infty} f_j$ the following statements are equivalent:
\begin{enumerate}[{\rm (i)}]
\item $f_0=\cdots=f_p=0$;
\item $\phi(f)=O(t^{p+1})$.
\end{enumerate}
\end{lemma}

Next we prove a key lemma:

\begin{lemma}\label{l:D}
For $X\in \pow S\inspitz t$ and  element $C$ in $\pow S$ of degree one the following statements are equivalent:
\begin{enumerate}[{\rm (i)}]
\item $X-e^{Ct}=O(t^{p+1})$;
\item $X(0)=1$ and $D:=\dot X-CX$ enjoys $D=O(t^p)$.
\end{enumerate}
\end{lemma}

\begin{proof}
Certainly (i) implies (ii), as can be seen by differentiation. Conversely, if (ii) holds, one only needs to check that $X-e^{Ct}$ vanishes when setting $t=0$.
But this is a consequence of the assumption that $X(0)=1$.
\end{proof}

Here is the main observation about the different method to be used in Section~\ref{s:splitting-schemes} on {\em splitting schemes}. Namely it will imply that for deriving order conditions it is equivalent to either
consider them as  the coefficients of a power series or to pass to the logarithm and use thereby the BCH-formula and
look at the coefficients of the basic commutators.

\begin{proposition}\label{p:maintool}
Suppose that $h\in\pow S$ has the form $h=e^Z$ with $Z_0=0$.
Then the following statements about $h=\sum_{j=0}^{\infty} h_j$ and $e^C$ for an element $C$ homogeneous of degree 1 are equivalent:
\begin{enumerate}[{\rm (A)}]
\item $h_0-1=h_1-C=\cdots =h_p-\frac{C^p}{p!}=0$;
\item $\phi(h)^{\cdot}-C\phi(h)=O(t^p)$;
\item $Z-C=O(t^{p+1})$.
\end{enumerate}
\end{proposition}

\begin{proof}
As $\phi(h)=\sum_{j=0}^{\infty} h_jt^j$, the equivalence of (A) and (B) follows from Lemma~\ref{l:D}.
The equivalence of (B) and (C) is an immediate consequence of Corollary~\ref{c:homog}.
\end{proof}

\section{The Baker-Campbell-Hausdorff Formula}

The {\em Baker-Campbell-Hausdorff} formula (BCH-formula)
(see for instance~\cite{haireretal06}) allows,
for given $X$ and $Y$ in $\alg S$  without constant terms,
to find $Z\in\pow S$ with $e^Xe^Y=e^Z$. In fact, $Z$ turns out to be a formal infinite sum of $X$, $Y$ and
homogeneous elements from the Lie algebra $\lie S$, generated by the set $S$ and  the bracket operation
$[l_1,l_2]=l_1l_2-l_2l_1$ for $l_i\in\lie S$.

\begin{example}\label{ex:BCH2}
The first terms of $Z$ are
\[Z=X+Y +\frac{1}{2}[X,Y] + \frac{1}{12}\big([X,[X,Y]]+[Y,[Y,X]]\big)+\cdots{} \]
\end{example}

As noted $Z$,
with the exception of the terms of first order,
is an infinite sum of Lie elements, i.e., homogeneous
elements of the Lie algebra generated by $X$ and $Y$. One  denotes $Z$ by $\log{e^Xe^Y}$.
Inductively one can derive an analog for $\log(e^{X_1}\cdots e^{X_n})$, for $X_i\in S$.

\section{Splitting Schemes}
\label{s:splitting-schemes}


The following abstract definition of a {\em splitting scheme} will serve our purpose:

\begin{definition}\label{d:splitting-scheme}\rm
Given $A$ and $B$ in $S$ and suppose there are, for $j=1,\ldots,s$,
$A_j\in\linkomb A$ and $B_j\in\linkomb B$,
i.e., $ A_j = a_j A $, $ B_j = b_j B $ with scalar coefficients
$ a_j,b_j, j=1,\ldots,s $.
Then these data determine a  {\em splitting scheme} of order at least $p$, provided
\[e^{A_1t}e^{B_1t}\cdots e^{A_st}e^{B_st}-e^{(A+B)t}=O(t^{p+1}) \]
\end{definition}

Here is an equivalent formulation of this condition. The proof, in light of the BCH-formula, is immediate from the definition:

\begin{proposition}\label{p:splitting-scheme}
The following statements for given $A$ and $B$ in $S$ and elements $A_i\in \linkomb A$, $B_i\in\linkomb B$,
where $i=1,\ldots,s$,
are equivalent:
\begin{enumerate}[{\rm (A)}]
\item The data yield a splitting scheme of order at least $p$;
\item $\log(e^{A_1}e^{B_1}\cdots e^{A_s}e^{B_s}) - (A+B)$ has homogeneous terms equal to zero for $j=1,\ldots,p$.
\end{enumerate}
\end{proposition}

\begin{remark}
Splitting techniques can also also successfully applied to nonlinear evolution equations.
The order conditions studied here are also valid for this general case.
This follows from an ingenious idea by W.~Gr{\"o}bner~\cite{groebner60},
namely formally to express the flow of a nonlinear evolution equation as the
exponential of the corresponding Lie derivative; see~\cite[Section~III.5]{haireretal06}.
\end{remark}

Let us, as a preparation for Section~\ref{s:examples},
compute the logarithm in (A) for $s=2$ and $s=3$ up to terms of order $p\le3$.

The BCH-formula easily yields
\begin{equation}\label{eq:bch-ab}
\log(e^{A_j}e^{B_j})=A_j+B_j+\frac12\co ABjj+\frac1{12}\left(\cco AABjjj+\cco BBAjjj\right)+O(4)
\end{equation}
where $O(4)$ stands for all terms in $\lie S$ of degree at least $4$.

\begin{example}\label{ex:xy}
Let $X=X_1+X_2+X_3$ and $Y=Y_1+Y_2+Y_3$ be a decomposition into homogeneous
elements with all nonlinear terms in $\lie S$.
Then $\log(e^Xe^Y)=H$ has first homogeneous terms

\begin{align*}
H_1 &= X_1+Y_1\\
H_2 &= X_2+Y_2+\frac12\co xy11\\
H_3 &= X_3+Y_3+\frac12\left(\co xy12+\co xy21\right) \\
    & \qquad\qquad\quad {} +\frac1{12}\left(\cco xxy111 + \cco yyx111\right)
\end{align*}
\end{example}

\renewcommand{\co}[2]{[#1,#2]}
\renewcommand{\cco}[3]{[#1,[#2,#3]]}

\begin{lemma}\label{l:up2-3} Let $X:=aA+bB+c\co AB+d\cco ABB +e\cco BBA$ and
$X':=a'A+b'B+c'\co AB+d'\cco ABB +e'\cco BBA$, then $H:=\log(e^Xe^Y)$ has first terms
$$H=H_1A+H_2B+H_3\co AB+H_4\cco AAB+H_5\cco BBA$$
where
\[
\begin{array}{lcl}
H_1&=&a+a'\\
H_2&=&b+b'\\
H_3&=&c+c'+\frac12(ab'-a'b)\\
H_4&=&d+d'+\frac12(ac'-a'c)+\frac1{12}(ab'-a'b)(a-a')\\
H_5&=&e+e'-\frac12(bc'-b'c)-\frac1{12}(ab'-a'b)(b-b')\\
\end{array}\]
\end{lemma}

\begin{proof}
Using the preceding example for $Y:=X'$ and elementary computation yield the result.
\end{proof}

\newcommand{\cc}[2]{[#1,#2]}
\newcommand{\ccc}[3]{[#1,[#2,#3]]}
\newcommand{\cab}[3]{\ccc{\ab#1}{\ab#2}{\ab#3}}
\newcommand{\ab}[2]{a_{#1}b_{#2}}
\newcommand{\aab}[3]{a_{#1}a_{#2}b_{#3}}
\newcommand{\bba}[3]{b_{#1}b_{#2}a_{#3}}

\renewcommand{\aa}[3]{#1_{#2}^{#3}}

With the aid of Lemma \ref{l:up2-3} one finds:

\begin{corollary}\label{c:K123} The first homogeneous terms of $K:=\log(e^{a_1A}e^{b_1B}e^{a_2A}e^{b_2B}e^{a_3A}e^{b_3B})$ are as follows:%
\begin{align*}
K_1 &= (a_1+a_2+a_3)A+(b_1+b_2+b_3)B\\
K_2 &= \frac12\left(\ab11+\ab22+\ab33+\ab12\right.\\
       &  \quad {} \left.-\ab21+\ab13-\ab31+\ab23-\ab32\right)\cc AB
\end{align*}

and $K_3=\xi \ccc AAB + \eta \ccc BBA$ with

\begin{align*}
\xi &= \frac1{12}(\aa a12b_1+\aa a22b_2+\aa a33b_3)\\
    & \quad {} + \frac14(\aab122-\aab121+\aab133+\aab233-\aab131-\aab232-\aab132+\aab231)\\
    & \quad {} + \frac1{12}(\aa a12b_2-\aab121+\aa a22b_1-\aab122+\aa a12b_3-\aab131+\aab123-\aab132\\
     & \quad {} + \aab123-\aab231+\aa a22b_3-\aab232\\
     & \quad {} + \aa a32b_1-\aab133+\aa a32b_2-\aab233)\\[2mm]
\eta &= \frac1{12}(\aa b12a_1+\aa b22a_2+\aa b33a_3)\\
    & \quad {} + \frac14(\bba122-\bba121+\bba133+\bba233-\bba131-\bba232-\bba132+\bba231)\\
    & \quad {} + \frac1{12}(\aa b12a_2-\bba121+\aa b22a_1-\bba122+\aa b12a_3-\bba131+\bba123-\bba132\\
     & \quad {} + \bba123-\bba231+\aa b22a_3-\bba232\\
     & \quad {} + \aa b32b_1-\bba133+\aa b32a_2-\bba233)
\end{align*}
\end{corollary}

\begin{proof}
One first computes
$s_j:=\log(e^{a_jA}e^{b_jB})=a_jA+b_jB+\frac12a_jb_j[A,B]+\frac1{12}(a_j^2b_j[A,[A,B]]+a_jb_j^2[B,[B,A]]+O(4)$.
Then, using Lemma \ref{l:up2-3}, compute first $H:=\log(e^{s_1}e^{s_2})$ and, again using the lemma,
find the desired expressions by computing $K:=\log(e^He^{s_3})$.
\end{proof}

\section{Computing Order Conditions in Examples}\label{s:examples}
\renewcommand{\ab}[2]{a_{#1}b_{#2}}

\subsection{Schemes of order at least $1$}

It follows right from the definition that in this case
\[\sum_{j=1}^sA_j=A, \ \ \sum_{j=1}^sB_j=B \]
i.e., for $A_j=a_jA$ and $B_j=b_jB$ one obtains
\[\sum_{j=1}^sa_j=1, \ \ \sum_{j=1}^sb_j=1\]

\subsection{The order conditions for $s=2$ and $p=3$}
Elementary computation leads to the following observation:

\begin{lemma}\label{l:s2p3} The order conditions for $s=2$ and $p=3$ are as follows:
\begin{align*}
1 &= a_1+a_2\\
1 &= b_1+b_2\\
0 &= \ab11+\ab22+\ab12-\ab21\\
0 &= 1-6a_1a_2b_1  \\
0 &= 1-6b_1b_2a_2
\end{align*}
\end{lemma}
\begin{proof} The order condition for $s=1$ in the first line follows from the previous subsection.
The higher order conditions result
from Lemma \ref{l:up2-3}.
\end{proof}

\subsection{The order conditions for $s=3$ and $p= 3$}
Making use of Corollary~\ref{c:K123} the conditions on the coefficients $a_j$ and $b_j$ for $j=1,2,3$ in
order to let $A_j=a_jA$ and $B_j=b_j$ determine the necessary and sufficient conditions
for a splitting scheme of order $p$ at least 3 when $s=3$.
\newcommand{\ba}[2]{b_{#1}a_{#2}}

\begin{lemma}\label{l:o3}
The order conditions for $s=p=3$ read as follows:
\begin{align*}
1 &= a_1+a_2+a_3\\
1 &= b_1+b_2+b_3\\
\frac12 &= \ab21+\ab31+\ab32\\
2 &=  3(a_2+a_3)-6\aab232 \\
2 &= 3(b_1+b_2)-6\bba122
\end{align*}
\end{lemma}

\begin{proof}
In Corollary~\ref{c:K123}, one equates the coefficients of $A$ and $B$ to $1$, and those of
$[A,B]$, $\ccc AAB$ and $\ccc BBA$ to zero. Then, using the third equation, terms
 $a_3b_2$ have been eliminated from the last two equations.
\end{proof}

To conclude this section let us remark that developing $e^{X_1+X_2+\cdots}$ as a Taylor series, one
finds from Lemma \ref{l:homog} and Proposition \ref{p:maintool}:

\begin{proposition}\label{p:lie-terms}
The following statements for an exponential function $e^X$ for $X$ a sum of homogeneous Lie elements
with the exception of the linear term are equivalent:
\begin{enumerate}[{\rm (i)}]
\item $X_1=X_2=\cdots=X_p=0$
\item As a power series in $\pow S$ the first non vanishing homogeneous term is $\frac1{(p+1)!}X_{p+1}$.
\end{enumerate}
\end{proposition}

\section{Alternative approach via Taylor expansion and computation in the free Lie algebra generated by $ A,B $}\label{s:lyndon}
According to the ideas from~\cite{AH14,auzingeretal13b}
systems of polynomial equation representing order conditions for splitting methods
are set up in a different way
without making explicit use of the BCH formula.
This is straighforward to implement in computer algebra.
The resulting systems of equations are not identical
but equivalent to those obtained when implementing
the BCH-based procedure described above.

This alternative approach described can also easily be adapted and generalized to
cases with various symmetries, pairs of schemes,
and more general cases like splitting involving three operators $ A,B,C $, or more,
see~\cite{auzingeretal13b}.

To find conditions for the coefficients $ a_j,b_j $ such that for $ A_j=a_j A $, $ B_j = b_j B $
a scheme of order $ p $ is obtained,
\[L(t) := e^{A_1 t}e^{B_1 t}\cdots e^{A_s t}e^{B_s t}-e^{(A+B)t}=O(t^{p+1}), \]
we consider the Taylor expansion of $ L(t) $, the local error of the
splitting scheme applied with stepsize $ t $ (satisfying $ L(0)=0 $
by construction),\footnote{Here, $ \ddt{q}L(0) := \ddt{q}{L}(t)\big|_{t=0} $.}
\[
L(t) =
\sum_{q=1}^{p} \frac{t^q}{q!}\ddt{q}L(0) + O(t^{p+1}).
\]
The method is of order $ p $ iff $ L(t) = O(t^{p+1}) $;
thus the conditions for order $ p $ are given by
\begin{equation} \label{OC-Tay}
\ddt{}L(0) = \ldots = \ddt{p}L(0) = 0.
\end{equation}
Via successive differentiation of $ L(t) $ we obtain the following homogeneous representation
of $ \ddt{q}L(0) $ in terms of power products of in the non-commuting variables $ A $ and $ B $:
With $ {\bm k}=(k_1,\ldots,k_s) \in \NN_0^s $,
\begin{equation} \label{OC-Tay-1-AB}
\ddt{q}L(0)
= \sum_{|{\bm k}|=q} \dbinom{q}{{\bm k}}
  \prod\limits_{j=1}^{s}
  \sum_{l=0}^{k_j} \dbinom{k_j}{l}A_j^{l}B_j^{k_j-l}
  \;-\; (A+B)^q, \quad q = 0,1,2,\ldots
\end{equation}
In a computer algebra system, these symbolic expressions can be generated in a straightforward eay.

If conditions~\eqref{OC-Tay} are satisfied up to a given order $ p $,
then the leading term of the local error is given by $ \frac{t^{p+1}}{(p+1)!}\ddt{p+1}L(0) $.
Proposition~\ref{p:lie-terms} shows
that this leading local error term is a homogeneous linear combination of Lie elements.
With the terminology
\begin{center}
$ \text{LC}(q) $ := `the sum~\eqref{OC-Tay-1-AB} is a linear combination of
                     Lie elements of degree $ q $'
\end{center}
this amounts to $ \text{LC}(p+1) $ being true for a scheme of order $ p $.
Exploited this statement allows to design a recursive algorithm for generating a set of
order conditions:
\begin{itemize}
\item[(i)] By construction, $ \text{LC}(1) = 0 $ holds.
      But, {\em a~priori}, for $ q>1 $ the expression~\eqref{OC-Tay-1-AB} for $ \ddt{q}L(0) $
      does {\em not}\, enjoy $ \text{LC}(q) $, see Example~\ref{s=3-p=3} below.
      On the other hand, from Proposition~\ref{p:lie-terms} we know that
      \[
      \text{LC}(1) ~\land~ \ldots ~\land~ \text{LC}(q-1) ~~\Rightarrow~~ \text{LC}(q).
      \]
      By induction over $ q $ we see that each solution of~\eqref{OC-Tay}
      must satisfy $ \text{LC}(q) $ for $ q=1,\ldots,p $.
      (Moreover, for each the resulting solution $ \text{LC}(p+1) $ will hold true.)
\item[(ii)] Due to~(i), for the purpose of solving the system~\eqref{OC-Tay} we may assume
      \[
      \ddt{q}L(0) = \sum_k \lambda_{q,k}\,B_{q,k}, \quad q=1,\ldots,p
      \]
      where the $ B_{q,k} $ of degree $ q $ are elements from a basis of the free
      Lie algebra generated by $ A $ and $ B $. Now the problem is to identify,
      on the basis of the expressions~\eqref{OC-Tay-1-AB},
      coefficients $ \lambda_{q,k}=\lambda_{q,k}(a_j,b_j) $ such that the polynomial system
      \begin{equation} \label{OC-lambda}
      \quad\lambda_{q,k} = 0, \quad q=1,\ldots,p, ~~
      \text{$ k $ running over all basis elements of degree $ q $}
      \end{equation}
      will be equivalent to~\eqref{OC-Tay}.
      To this end we make use of the one-to-one correspondence between
      basis elements of degree $ q $
      represented by non-associative, bracketed words (commutators)
      and associative words of length $ q $ over  over the alphabet $ \{ A,B \} $.
      The implementation described in~\cite{AH14,auzingeretal13b}
      relies on the Lyndon basis, also called Lyndon-Shirshov basis,
      which can be generated by an algorithm devised in~\cite{duval88}.
      With this choice,
      \begin{equation} \label{lexord}
      \begin{aligned}
      \qquad
      &\text{\em each basis element $ B_{q,k} $ is uniquely represented by a} \\[-0.75mm]
      &\text{\em Lyndon word of length $ q $ associated with the leading term,} \\[-0.75mm]
      &\text{\em in lexicographical order, of the expanded version of $ B_{q,k} $.}
      \end{aligned}
      \end{equation}
      Identifying the coefficients of these `Lyndon monomials' results in the
      desired polynomial system~\eqref{OC-lambda}.

      To be more precise we note that,
      in general, a given Lyndon monomial shows up in different expanded
      commutators. Therefore, equating coefficients of Lyndon monomials
      will not directly result in the system~\eqref{OC-lambda}
      but, as a consequence of~\eqref{lexord}, in an equivalent system
      which is obtained from~\eqref{OC-lambda} by premultiplication with a regular
      triangular matrix.
\end{itemize}
For the underlying theoretical background concerning Lyndon bases in free Lie algebras
we refer to~\cite{bokutetal2006}.
For a detailed illustration of our approach for order $ p=5 $
see Example~2 from~\cite{auzingeretal13b}.
In the following example we reconsider the case $ s=p=3 $.

\begin{example}\label{s=3-p=3}
For $ s=3 $ we have (see~\eqref{OC-Tay-1-AB})
\begin{align*}
\ddt{} L(0) &= \big(a_1+a_2+a_3-1\big)A + \big(b_1+b_2+b_3-1\big)B,  \\
\ddt{2}L(0) &= \big((a_1+a_2+a_3)^2-1\big)A^2  \\
& \quad {} + \big(2a_1(b_1+b_2+b_3) + 2a_2(b_2+b_3) + 2a_3b_3 - 1\big)AB \\
& \quad {} + \big(2a_2b_1 + 2a_3(b_1 + b_2) - 1\big)BA \\
& \quad {} + \big((b_1+b_2+b_3)^2-1\big)B^2,
\end{align*}
Assume that $ a_1+a_2+a_3=1 $ and $ b_1+b_2+b_3=1 $ such that $ \ddt{} L(0) = 0 $.
Then substituting $ a_1=1-a_2-a_3 $ and $ b_3=1-b_1-b_2 $ into $ \ddt{2}L(0) $ gives
the commutator expression
\[ \ddt{2}L(0) = -\big(2a_2b_1 + 2a_3(b_1 + b_2) - 1\big)[A,B]. \]
Therefore the system
\begin{align*}
a_1 + a_2 + a_2 &= 1 \\
b_1 + b_2 + b_3 &= 1 \\
a_2b_1 + a_3(b_1 + b_2) &= \frac{1}{2}
\end{align*}
represents a set of conditions for order $ p=2 $.

Extending this computation to $ p=3 $ by hand is already somewhat laborious.
However, from the above considerations
we know that assuming the conditions for order $p=2 $ are satisfied, then
\[ \ddt{2}L(0) = \lambda_{AAB}[A,[A,B]] + \lambda_{ABB}[[A,B],B], \]
where
\begin{align*}
\lambda_{AAB} &= \text{coefficient of the power product $ A^2 B $ in the expression~\eqref{OC-Tay-1-AB} for $ \ddt{2}L(0) $}, \\
\lambda_{ABB} &= \text{coefficient of the power product $ A B^2 $ in the expression~\eqref{OC-Tay-1-AB} for $ \ddt{2}L(0) $}.
\end{align*}
Here the two independent commutators $ [A,[A,B]] $ and $ [[A,B],B] $ are represented by the
associative Lyndon words "AAB" and "ABB".
In computer algebra, extraction of the coefficients $ \lambda_{AAB} $ and $ \lambda_{ABB} $
from the symbolic expression
$ \ddt{2}L(0) $ is straightforward. In this way we end up with the system
\begin{align*}
a_1 + a_2 + a_3 &= 1\\
b_1 + b_2 + b_3 &= 1 \\
a_2b_1 + a_3(b_1 + b_2) &= \frac{1}{2} \\
a_2b_1^2 + a_3\big(b_1+b_2\big)^2 &= \frac{1}{3} \\
{(a_2+a_3)}^2b_1 + a_3^2b_2 &= \frac{1}{3}
\end{align*}
representing a set of order conditions for order $ p=3 $. The system is equivalent to the one found in Lemma \ref{l:o3}.
We note that there is a one-dimensional zero solution manifold
containing well-known rational solutions, e.g.,
\begin{align*}
& a_1 = \frac{7}{24},~ a_2 =  \frac{3}{4},~ a_3 = -\frac{1}{24}, \\
& b_1 = \frac{2}{3},~  b_2 = -\frac{2}{3},~ b_3 = 1.
\end{align*}
\end{example}

\end{document}